\documentclass[12pt]{article}
\usepackage{amssymb}
\usepackage{graphicx}
\begin{document}
\bibliographystyle{plain}
\setlength{\baselineskip}{1.5\baselineskip}
\newtheorem{theorem}{Theorem}
\newtheorem{lemma}{Lemma}
\newtheorem{corollary}{Corollary}
\newcommand{\eqref}[1]{(\ref{#1}) }
\newcommand{\boldm}[1]{\mbox{\boldmath$#1$}}
\newcommand{\bigO}[1]{{\mathcal{O}}\left( {#1 }\right)}
\newcommand{\bigOp}[1]{{\mathcal{O}}_p\left( {#1} \right)}
\newcommand{\smallO}[1]{{o}\left( {#1} \right)}
\newcommand{\sign}[1]{{sign}\left( {#1} \right)}
\newcommand{\smallOp}[1]{{o}_p\left( {#1} \right)}
\renewcommand{\baselinestretch}{1}
\newenvironment{proof}[1][Proof]{\noindent\textbf{#1.} }{\ \rule{0.5em}{0.5em}}

\title{Exact Probability Bounds under 
Moment-matching Restrictions} 
\author{Stephen Portnoy{$ ^1 $}  }
\date{\today}

\maketitle 

\bigskip

\begin{abstract}

Lindsay and Basak (2000) posed the question of how far from 
normality could a distribution be if it matches $k$ normal moments. They provided a bound on the maximal
difference in c.d.f.'s, and implied that these bounds were attained. It will be shown here that in fact the bound
is not attained if the number of even moments matched is odd. An explicit solution is developed as a symmetric
distribution with a finite number of mass points when the number of even moments matched is even, and this bound 
for the even case is shown to hold as an explicit limit for the subsequent odd case.

\end{abstract}

\footnotetext[1]{\noindent 
Professor, Department of Statistics, University
of Illinois at Urbana-Champaign \\
\smallskip
$\quad$ corresponding email: sportnoy@illinois.edu \\ }

\medskip

\newpage

Portnoy (2015) presents results partially correcting claims in Lindsay and Basak (2000) concerning the 
worst-case approximation of a normal distribution by a distribution that matches a given number of 
moments. The formal mathematical statements and proofs are given here.

\bigskip

\begin{theorem} \label{Vandersol}
Let $\{ x_1 , \,  \cdots \, , \,  , x_n \}$ be any domain 
of different non-zero values with
associated probabilities $\{ p_1 , \,  \cdots \, , \,  , p_n \}$.
Suppose the moments are matched
\begin{equation} \label{momeq}
\sum_{i=1}^n p_i x_i^j = M_j  \qquad  j = 1 , \, \cdots \, , \ n 
\end{equation}
where
\begin{equation} \label{normom}
M_\ell \equiv E \, Z^\ell \,\, , \qquad  Z \sim {\cal{N}}(0, \, 1) \,\, .
\end{equation}
Then, for $\, j = 1 , \, \cdots \, , \ n \,$,
\begin{equation} \label{pjdef}
p_j = \frac{ \sum_{i=1}^n (-1)^i \, M_{n-i+1} \,  e_{i-1}(\sim  x_j) } 
   { \prod_{i=1}^n \, (x_i - x_j) } 
\end{equation}
where $\, e_m(y_1 , \, \cdots Ê\, , \, \ y_n)  \,$ denotes the $m$th elementary 
symmetric function of its arguments, and the argument $( \sim  y_j)$ denotes the
$(n-1)$-vector $\, (y_1 , \, \cdots Ê\, , \, \ y_n) \,$ with $\, y_j \,$ deleted. Furthermore,
\begin{equation} \label{sumpsol}
r \equiv \sum_{j=1}^n  p_j = \frac{ \sum_{i=1}^n (-1)^i \, M_{n-i+1} \,
   e_{i-1}(x_1, \, x_2, \, . . . \, , \, x_n ) } 
   { \prod_{i=1}^n \, x_i } \,\, .
\end{equation}
\end{theorem}
     
\begin{proof}
Consider the Vandermonde matrix
\begin{equation} \label{Vandef}
V \equiv \left( \begin{array}{c c c c c} 
1   & 1          &   1        &  \cdots   & 1          \\
0   & x_1      & x_2      & \cdots   & x_n       \\
0   & x_1^2  & x_2^2  & \cdots  & x_n^2   \\
\cdots &  \cdots & \cdots  & \cdots   & \cdots  \\ 
0   &  x_1^n & x_2^n  & \cdots  & x_n ^n        
 \end{array}  \right)
\end{equation}
Introduce $\, p_0 \equiv 1 - r = 1 - \sum_{j=1}^n  p_j \, $. Then with $p$ denoting the $n$-vector 
with coordinates $p_j$, equation \eqref{momeq} yields the matrix equations:
\begin{equation} \label{Vaneqs}
  V \left( \begin{array}{c} p_0 \\ p \end{array} \right) = \left( \begin{array}{c}  1 \\  M \end{array} \right) 
  \qquad  \qquad   V_{22} \, p = M \, ,
  \end{equation}
 where $V_{22}$ is the 
lower left $\, n \times n \,$ submatrix of $V$ and $M$ is the vector of moments. Note also that 
the argument here does not require $M$ to be the moments of a Normal distribution: 
any vector will provide the same formulas, though I have no  general result providing conditions
under which the $p_j$'s solving \eqref{momeq} need be in $[ 0 , \, 1 ]$.

Eisinberg and Fedele (2006) provide a formula for the inverse element of $V$ (where
in the notation of that paper, we have taken $\, x_0 = 0$). Specifically,
for  $\, i = 0 , \, 1 , \,  \cdots \, , \, , n \,$:

$$
 (V^{-1})_{ij} =   \phi_{nj} \, \Psi_{n1i}    
 $$
 where
$$
\phi_{nj} = 1 / \prod_{k=1}^n (x_k - x_j)
$$
from the recursion in equation (6) of Eisinberg and Fedele (2006) plus a direct induction argument; and
$$
\Psi_{n1i} = (-1)^{i+1} \, e_{n+1-i}(\sim x_j)
$$
from equation (26) of that paper.

Thus, noting that $\, V_{11}^{-1} = 1 \,$, the first row of \eqref{Vaneqs} yields
\begin{eqnarray} 
r = \sum_{j=1}^n p_j  & = & \sum_{i=1}^n (-1)^{i+1} \, M_i 
   e_{n-i}(x_1 , \, . . . \, , \, x_n) / \prod_{j=1}^k (x_j - x_0) \\
  &  =  &  \sum (-1)^i \, M_{n-i+1} 
   e_{i-1}(x_1 , \, . . . \, , \, x_k) / \prod_{j=1}^n (x_j) \,\, .
\end{eqnarray}

Note that the block-triangular form for $V$ shows that the formula above also gives the
corresponding elements of $\, V_{22}^{-1} $. Thus, as above, the second row
of matrix equation  \eqref{Vaneqs} immediately  provides \eqref{pjdef}. 
\end{proof}

\bigskip
\bigskip

To show that it suffices to consider symmetric discrete distributions,

It is possible to simplify computations by considering only the positive values in a symmetric domain.
Specifically, let $\, \{ \, 0 < y_1 < \, \cdots \, < y_n \, \} \, $ be any domain of positive mass points.
Consider the symmetric mixture with (positive) probabilities $\, p_j/2 \,$
at each of $\, \pm y_j \,$ and probability $\, p_0 \,$ at zero. Suppose the $\, p_j$'s 
generate $m$ even moments:
$$
\sum_{j=1}^n \, y_j^{2i} = M_{2i} / 2  \equiv E Z^{2i} / 2  \qquad i = 1 , \, \cdots \, , \, m \,\, .
$$
Define
\begin{equation} \label{r*def}
r^* \equiv \sum_{j=1}^n p_j \,\, .
\end{equation}
Then, the c.d.f. of the mixture at zero is just $\, F(0) = 1 - r^* \,$ and $\, p_0 = 1 - 2 r^* \,$. Thus,
maximizing the c.d.f. at zero is the same as minimizing $r^*$.

\bigskip

The following result shows that symmetric distributions suffice. Since the result is not
required for the subsequent optimality results, the proof is only sketched.

\begin{theorem} \label{thm2}
Let $\, \epsilon > 0 \,$ be given. Let
$\{ x_1 , \,  \cdots \, , \,  , x_n \}$ be any domain for which there are (non-zero) 
associated probabilities with moments matching $k$ normal moments and 
whose c.d.f. at zero is within $\epsilon$ of the maximal value 
(over all distributions with $k$ matching moments). Then 
there is a set of symmetric points, $\{ \pm y_j \, : \, y_j > 0 \}$, 
with associated positive probabilities also matching $k$ Normal moments 
and for which the c.d.f. at zero is also within $\, \epsilon \,$ of its maximal value.
\end{theorem}.

\begin{proof} 
Choose a perturbation of the domain, $ \{ x_i^* \}$ so that all $\, | x_i^* | \,$
are different and so that the (associated probabilities) $\, p_i^* \,$
satisfying the moment equalities are all positive are such that the
c.d.f at zero is within $\, \epsilon \,$  of the optimum. Introduce the symmetrized domain
$\, \{ \{ x_i^* \} , \, \{ - x_i^* \}  \}  \,$. Consider the linear programming problem
given by equations (1) and (2) of Portnoy (2015) applied to the symmetrized domain.
Since the original distribution in the statement of the Theorem
provides a feasible solution, the linear programming solution provides an optimal value
that is also within $\epsilon$ of the original optimal value. By symmetry, the solution
to the liner programming problem is also symmetric.
\end{proof}

\bigskip

Theorem \ref{oddNoSol} (below) requires the derivatives of the probabilities in \eqref{pjdef} (with respect to the
mass points), which can be computed easily:

\begin{lemma} \label{sgnAlt}
Consider a set of positive mass points $\, \{ 0 <  y_1 \, < \, y_2 \, < \, \cdots \, < \, y_n \} \,$. 
Let $p_j$ be the corresponding probabilities matching moments and let $r^*$ be the sum 
of probabilities given by \eqref{sumpsol} (see Theorem \ref{Vandersol}). Then
$\, \frac{\partial r^*}{\partial y_j} \,$ alternates in sign for $\, j = 1 , \, , \cdots \, , \, n \,$, with 
the partial derivative with respect to $y_1$ being positive.
\end{lemma}
\begin{proof}
For each $\, j \,$, factor $\, 1 / y_j \,$ from $\, r^* \,$:
\begin{eqnarray} \label{roveryj}
r^* & = & \frac{1}{y_j} 
 \frac{ \sum_{i=1}^k (-1)^i \, M_{k-i+1} \,  e_{i-1}^k(- y_j) } { 2 \, \prod_{i \ne j} \, y_i } 
  \, + \, g(\sim y_j)  \\
& = &  p_j^* \frac{ \prod_{i \neq j} \, (y_i - y_j) } { \prod_{i \neq j} \, y_i }  \, + \, g(\sim y_j)
\end{eqnarray}
where the term $ \, g(\sim y_j) \,$  does not depend on $\, y_j \,$ (since each of the 
numerator terms in $g$ has a factor of $y_j$ that is cancelled by the $y_j$ factor 
in the denominator product). Since the $\, y_j$'s are positive, the product 
$\,  \prod_{i \neq j} \, (y_i - y_j)  \,$ alternates in sign,
and the theorem follows.
\end{proof}

\bigskip
\bigskip

Finally, the basic optimality results are presented.
First, assume the number of even moments, $k$, is even, 
and let  $\{ y_1 , \, \, \cdots \, , \, y_{k/2} \}$
be the squares of the  non-zero roots of the Hermite polynomial $He(2k+1)$. 
Note that the $k$ non-zero roots of $He(2k+1)$ are located  symmetrically about
zero, and so there are only $k/2$ squares.

Let $M_j$ denote the normal moments (see \eqref{normom}).
By the standard theory for Gaussian quadrature and symmetry, $M_j$ is
given by twice the sum of the $k/2$ even gaussian quadrature weights 
times the $jth$ power of $y_i$ ($\, j = 1 , \, \cdots \, , \, k \, )$. Thus, the weights
can be determined by any $k/2$ even moment equalities. Let $p_j$
$( \, j = 1 , \, \cdots \, , \, k/2 \, )$ satisfy:
\begin{equation}
\sum_{i=1}^{k/2} p_i y_i^{2j} = M_{2j}/2 \qquad  j = 1 , \, \cdots \, , \, k/2 \, .
\end{equation}

That is, the $p_j$'s are just the even weights. Since these are known to be positive 
and sum to less than one, they can define a discrete probability distribution by 
symmetrization and introduction of $p_0$. The following theorem shows that this
distribution is least favorable in the sense of maximizing $p_0$ (or, equivalently,
maximizing the difference from the normal c.d.f. at zero).

\begin{theorem} \label{p0EQlb} 
If the number of even moments, $k$,  is even, then the solution
described above achieves 1/2  of the bound given in Theorem 2
of Lindsay (2000), and therefore is least favorable.
\end{theorem}
\begin{proof}

From \eqref{sumpsol} the sum of the $p_j$'s is 
\begin{equation} \label{sumnum}
 r* = (1 - s/M_{2k})/2 
 \end{equation}
  where 
\begin{eqnarray*} 
s & = &  \sum_{i=1}^k (-1)^i \, M_{2(k-i+1)} \,  e_{i-1}^k (y_1, \, y_2, \, . . . \, , \, y_k )  \\
 & = & \sum_{i=1}^{k}  He[2k+2]_{i+2} M_{2i}  
 \end{eqnarray*}
Here,  $M_\ell$ is given by \eqref{normom} above, and we use the fact that 
the coefficients of  the Hermite polynomial $He[n]$ are just the elementary symmetric functions 
of the squared roots, $\{ y_j \}$ in opposite order. 

\medskip

Let $M$ be the Hankel moment matrix of even moments:
\begin{equation} \label{Hankel}
M = \left( \begin{array}{ c c c c c }
1          &  M_2   &  M_4  & \cdots &   M_{2k}     \\
M_2    & M_4   &  M_6   & \cdots &  M_{2k+2} \\
            &  \cdots &           &             &                     \\
M_{2k} &  M_{2k+2} & M_{2k+4} & \cdots &  M_{4k} 
\end{array} \right) \, .
\end{equation}

The upper bound from Lindsay(2000)  is  $\, 1 / M^{-1}_{1,1} \, $. As noted by
Lindsay, by symmetry of the normal distribution, if there is a distribution achieving
$\, .5 / M^{-1}_{1,1} \, $, this distribution must be least favorable (that is, it maximizes
the difference from the normal distribution evaluated at zero).

Let C be the matrix whose rows contain zeros and the non-zero coefficients
of the Hermite polynomials as follows:
\begin{equation} \label{Cmatrix}
\begin{array}{c}
He[2k+2]_{(2, \, 4 , \, \cdots \, , \,  2k+2)} \\
(0 \,\, , \, He[2k]_{(2, \, 4 , \, \cdots \, , \, 2k)} ) \\
(He[2k-1]_{(1, \, 3, \, \cdots \, , \, 2k-1) } , \,\, 0 ) \\
(0, \,\, 0, \,\, He[2k-2]_{(2, \, \cdots \, , 2k-2) } ) \\
(0,  \,\, He[2k-3]_{(1, \, 3, \, \cdots \, , \,2k-3) } , \,\, 0) \\
  \cdots   \\
(-1, \, 1,  \, 0, \, \cdots \, , \,  0)
\end{array}
\end{equation}

Note that 
\begin{equation} \label{zeroBYortho}
\sum_{j=0}^\ell   M_{2j+2i} \, He[2 \ell+2i +2]_{2j}
  =  \int_ x^{2i} \sum x^{2j} \, He[2 \ell+2i +2]_{2j}  \, \varphi(x)  dx  = 0
\end{equation}
by orthogonality of the Hermite polynomials.

Note also that the signs of the coefficients alternate, and
$He[2k]_2$ is the first nonzero coefficient and it equals $M_{2k}$.

Thus, using \eqref{zeroBYortho}, $D \equiv M C\, $  has first row and first column equal to
$ (  D_{1 1} , \,  0 \, , 0 , \, \cdots \, , \,0 ) $, where
\begin{equation} \label{D11}
D_{11} =   \sum_{i=0}^{k}  He[2k+2]_{i+2} \, M_{2i}  \, = \,   He[2k+2]_2 - s
\end{equation}
(from \eqref{sumnum}).

That is, $M C$ equals a partitioned block diagonal matrix with a $\, 1 \times 1 \,$
and a $(2k-1) \times (2k-1)$ submatrix. Thus, the inverse of $D$ is a similarly
partitioned block-diagonal matrix with first row and column
$$
(1/D[1,1] , 0 , 0 , . . . , 0) = ( 1 / (He[2k+2]_2 - s) , 0 , . . . , 0)
$$

Now since  $\, M = D C^{-1} \,$ , $\, M^{-1} = C D^{-1}$.
It follows that the upper (1, 1) element of $M^{-1}$ is  
$$
He[2k+2]_2 / (He[2k+2]_2 - s) = M_{2k} / (M_{2k} -s)
$$

Finally,  
$$
.5/M^{-1}[1,1] = .5 (M_{2k} - s)/M_{2k} = (1 - s/M_{2k})/2 \, ,
$$
which agrees with $r^*$  \eqref{sumnum}.
\end{proof}

\bigskip

Finally, consider the case when the number of matched even  moments is odd.

\begin{theorem} \label{oddNoSol}
If the number of matched even moments, $k$, is odd, then there is no distribution
matching the $k$ moments that maximizes the difference from the normal c.d.f. at zero. 
In fact, the maximum difference among moment-matching distributions 
approaches the maximal value for matching $(k-1)$ even moments, and is the limit
through a sequence of discrete mixtures whose maximal mass point, $y_k \rightarrow \infty$.
\end{theorem}
\begin{proof}

Assume there is a solution (to the symmetrized problem) with a finite number 
of mass points $\, {y_1 , y_2 , . . . y_n} \,$.  (Note, there is a solution
with $\, n=k$). If $\, n > k \,$, the moment equalities determine a manifold of
dimension $\,  n-k > 0 \,$, and so it must be possible to move at least one $y_j$
to make $p_0$ larger (since all partial derivatives are non-zero). Thus, 
if there is a solution, there is one with $k$ (odd) mass points.

Since all $p_j$'s are strictly positive for such a solution (in order
to match moments), $r^*$ would increase as $y_k$ increases,
since the derivative of $r^*$ with respect to $y_k$ is positive by Lemma \ref{sgnAlt}.
This contradicts the assumed optimality, and thus proves nonexistence of a
solution.

Note that  $\, p_k \leq M_k / y_k^k  \,$ (since $M_k$ is matched by positive values).
Thus, $p_k$ must decrease (to zero) as $y_k$ increases. Furthermore, 
for $j < k$ ,
$$
p_k \, y_k^j \leq M_k / y_k^{k-j}  \rightarrow 0, .
$$

Thus the first $k-1$ moments are nearly determined by  
$ \, \{ y_1, \, \cdots \, , \, y_{k-1} \} \,$, and (since $ p_k$ also tends to 0), the
optimal  $p_j^* $  is also (nearly) determined by the first $(k-1)$ $y_i$'s.
That is, the optimal $\, p_j^*$'s (for case $k$) are obtained as the limit (as  
$\, y_k \rightarrow \infty  \,$) of terms converging to  the optimal $p_j$'s for 
the $(k-1)$ case.
\end{proof}


\bigskip

\begin{thebibliography}{99}

\bibitem{EF}
Eisinberg and Fedele (2006). On the inversion of the Vandermonde matrix, {\it Applied mathematics and computation, 174}, 1384-1397.

\bibitem{Lind}
Lindsay, B. and Basak, P.  (2000). Moments determine the tail of a distribution (but not much else), {\it The
American Statistician}, 54, 248-251.

\bibitem{Por}
Portnoy, S. (2015). Maximizing Probability Bounds under 
Moment-matching Restrictions, {\it The American Statistician, 69}, 41-44.

\end{thebibliography}
\end{document}